\documentclass[12pt]{amsart}
\usepackage{amsfonts,amsmath,url}
\usepackage{amsthm}

\usepackage{amssymb}

\newtheorem{theorem}{Theorem}[section]
\newtheorem{prop}[theorem]{Proposition}
\newtheorem{lemma}[theorem]{Lemma}

\theoremstyle{definition}
\newtheorem{defn}{Definition}[section]
\newtheorem{rem}{Remark}[section]
\newtheorem{prob}{Problem}[section]

\newcommand{\grtype}{\mathop{\mbox{$\mathrm{A}$}}}
\newcommand{\partype}{\mathop{\mbox{$\mathrm{B}$}}}

\newcommand{\lcm}{\mathop{\mathrm{lcm}}}

\usepackage{amsthm}

\begin{document}

\title{Super graphs on groups, I}
\author[Arunkumar]{G. Arunkumar}
\address{Indian Institute of Technology, Dharwad, Karnataka, India.}
\email{arun.maths123@gmail.com}

\author[Cameron]{Peter J. Cameron}
\address{School of Mathematics and Statistics, University of St Andrews, Fife,
	UK.}
\email{pjc20@st-andrews.ac.uk}

\author[Nath]{Rajat Kanti Nath}
\address{Tezpur University, Tezpur, Assam, India.}
\email{rajatkantinath@yahoo.com}

\author[Selvaganesh]{Lavanya Selvaganesh}
\address{Indian Institute of Technology, Varanasi, India.}
\email{lavanyas.mat@iitbhu.ac.in}

\thanks{This work began in the Research Discussion on Groups and Graphs, organised by
Ambat Vijayakumar and Aparna Lakshmanan at CUSAT, Kochi, to whom we express
our gratitude. }


\date{}
\maketitle

\begin{abstract}
Let $G$ be a finite group. A number of graphs with the vertex set $G$ have been studied, including
the power graph, enhanced power graph, and commuting graph. These graphs form
a hierarchy under the inclusion of edge sets, and it is useful to study them
together. In addition, several authors have considered modifying the definition
of these graphs by choosing a natural equivalence relation on the group such as
equality, conjugacy, or equal orders, and joining two elements if there are
elements in their equivalence class that are adjacent in the original graph. In this way, we
enlarge the hierarchy into a second dimension. Using the three graph types and three equivalence relations mentioned gives nine graphs, of which in
general only two coincide; we find conditions on the group for some
other pairs to be equal. These often define interesting classes of groups,
such as EPPO groups, $2$-Engel groups, and Dedekind groups.

We study some properties of graphs in this new hierarchy. In particular, we
characterize the groups for which the graphs are complete, and
in most cases, we characterize the dominant vertices (those joined to all
others). Also, we give some results about universality,  perfectness, and clique
number.

The paper ends with some open problems and suggestions for further work.
\end{abstract}

\section{Introduction}

In this section, we describe some graphs associated with groups and discuss
a convention for these graphs as well as notation and terminology. In the
following sections, we prove a number of properties of the collection of graphs
 and show how treating them together can be helpful: we examine conditions
for some pair of these graphs to be equal; in most cases, we characterize the
dominant vertices; we show that some of the graphs are perfect, and examine
universality properties of the others; and we calculate the clique number in
some cases. 


\subsection{B superA graphs on groups}

\begin{defn}
	Let $\grtype$ be a type of graph defined on a group $G$. In this paper,
	we will consider three such types:
	\begin{enumerate}
		\item the \emph{power graph}, in which $g$ is adjacent to $h$ if either $g$ or $h$ is a power
		of the other;
		\item the \emph{enhanced power graph}, in which $g$ is adjacent to $h$ if $g$ and $h$ are both
		powers of an element $k$ (equivalently, if the group $\langle g,h\rangle$
		generated by $g$ and $h$ is cyclic);
		\item the \emph{commuting graph}, in which $g$ is adjacent to $h$ if $gh=hg$ (equivalently,
		if $\langle g,h\rangle$ is abelian).
	\end{enumerate}
\end{defn}

Several other types of graphs can be defined, including the deep commuting graph,
nilpotency graph, solvability graph, and Engel graph; these are described
in the survey paper~\cite{gong}.

Also, let $\partype$ be an equivalence relation defined on $G$. In this paper,
we will consider three equivalence relations;
\begin{enumerate}
\item \emph{equality}, $g\sim h$ if $g=h$;
\item \emph{conjugacy}, $g\sim h$ if $h=x^{-1}gx$ for some $x\in G$;
\item \emph{same order}, $g\sim h$ if $o(g)=o(h)$, where $o(g)$ denotes
the order of $g$.
\end{enumerate}
Other relations could be considered, such as \emph{automorphism conjugacy},
where $g\sim h$ if some automorphism of $G$ maps $g$ to $h$.

\paragraph{Definition} The \emph{$\partype$ super$\grtype$ graph} on $G$ is defined as
follows: Let $[g]$ denote the B-equivalence class of the
element $g$. Now join $g$ and $h$ if and only if there exist $g'\in[g]$
and $h'\in[h]$ such that $g'$ and $h'$ are joined in the $\grtype$-graph on $G$.

\medskip

In principle the graph and equivalence relation are arbitrary, but there are
reasons for choosing them to be preserved by the automorphism group of $G$,
as we will see. This is the case for the examples mentioned above.

\begin{rem}
If we take $\grtype$ to be the power graph and $\partype$ the relation ``same
order'', we do obtain the definition of the superpower
graph of $G$ from~\cite{ls:rdgg}. For suppose that $o(g)\mid o(h)$. Then some power of $h$,
namely $h^{o(h)/o(g)}$, has order $o(g)$; so in the definition we can take
$g'=h^{o(h)/o(g)}$ and $h'=h$. Our naming convention gives this graph the name
``order superpower graph'' of~$G$, which will distinguish this from the
conjugacy superpower graph. Also, Herzog \emph{et al.}~\cite{hlm} have
considered the conjugacy supercommuting graph. Our aim here is to give a
unified treatment of these graphs.
\end{rem}

Our convention would also give the power graph the name ``equality superpower
graph'', but we will simply say ``power graph'', with similar convention for
the other basic graph types.

We see that any result about the
power graph is in principle one of a set of nine results about related graphs.

Note that we have inclusions of the edge sets as follows: the edge set of the
power graph is contained in the edge set of the conjugacy superpower graph,
which is contained in the edge set of the order superpower graph; similarly
for other types of graphs.

Note also that graph parameters such as clique number, chromatic number, and
matching number are monotonic increasing with edge set; independence number
and clique cover number are monotonic decreasing; and increasing the edge
set cannot destroy properties such as being Hamiltonian.

\begin{prop}
\begin{enumerate}
\item Let the equivalence relation $\partype$ be ``same order''. If $g$ and $h$
are joined in the power graph, then for one of them, say $g$, every element
equivalent to $g$ is joined to some element equivalent to $h$.
\item Let $\partype$ be  the conjugacy relation and consider any graph of type
$\grtype$ on the group $G$.  If the graph $\grtype$ is invariant under inner automorphisms of $G$, and
$g$ is joined to $h$ in $\grtype$, then every element of the conjugacy
class of $g$ is joined to some element of the conjugacy class of $h$ in the
conjugacy super$\grtype$ graph on $G$, and \emph{vice versa}. 
\item More generally, let $H$ be a subgroup of the automorphism group of $G$ which acts on $\grtype$ for some graph type $\grtype$. Let $\partype$ be the
equivalence relation induced by the orbit partition of this action. Then $g$
is joined to $h$ in $\grtype$ implies that (in $\partype$ super$\grtype$ graph)
every element equivalent to $g$ is joined to some element equivalent to $h$ under the equivalence relation $\partype$.
\end{enumerate}
\label{p:orbit}
\end{prop}

\begin{proof}
The first statement was observed in the earlier remark. We prove the third statement from which second the statement follows. Note that  if $\{g,h\}$ is an edge,
then $\{\phi(g),\phi(h)\}$ is an edge for all $\phi \in H$ since $B$ is the equivalence relation induced by the orbit partition.
In particular, for the second statement, note that the conjugacy classes are orbits of the
inner automorphism group of $G$; so, if $\{g,h\}$ is an edge,
then $\{x^{-1}gx,x^{-1}hx\}$ is an edge for all $x\in G$. In this case, the hypothesis holds
for all the three graph types we are considering.
\end{proof}

\begin{rem}
Let us observe another general property. For each of the power graph,
enhanced power graph and commuting graph, if $H$ is a subgroup of $G$, then
$\grtype(H)$ is an induced subgraph of $\grtype(G)$. This holds also for the
order super$\grtype$ graphs, but not in general for the conjugacy
super$\grtype$ graphs since the conjugacy relation can change when we pass
from $G$ to $H$.
\end{rem}

\subsection{A convention about equivalence classes}

Our adjacency rule is ambiguous about whether we join vertices in the same
equivalence class. We now explain how we resolve this and explain the
rationale.

In many of the graphs defined on a group, including all those treated here,
the definition would naturally give us a loop at each vertex; any group 
element is a power of itself, so this holds for the power graph and enhanced
power graph; any element commutes with itself, so this holds for the commuting
graph; and so on. Of course, we prefer graphs not to have loops, so we
silently remove these, even though they make little difference to many
graph-theoretic properties (they make no change to connectivity and diameter,
and simply add the identity matrix to the adjacency matrix). Adopting the
convention that there is a (silent/virtual) loop at each vertex, we find that any
equivalence class of the equivalence relation B will induce a complete
subgraph in the $\partype$ super$\grtype$ graph.

In fact, even without this convention, things would not be very different.
Consider the order super$\grtype$ graph. An element $g$ has the same order as
its inverse, which is joined to it in the power graph, enhanced power graph, or
commuting graph; so as long as the order of $g$ is greater than~$2$, each
order equivalence class will induce a complete graph. This will fail only
for involutions. For the conjugacy relation, things are a bit more
complicated and could be worth investigating. However, as stated, we will
adopt the convention here that every equivalence class of $\partype$ induces a
complete subgraph in the $\partype$ super$\grtype$ graph.

\subsection{Notation}

We have defined a fairly large number of graphs: what notation should we use
to make it easy for the reader to recognize which graph is being discussed
without being altogether too cumbersome?

In \cite{gong}, the second author  proposed a systematic notation for various
graphs on groups: given a group $G$, that paper uses $\mathrm{Pow}(G)$,
$\mathrm{EPow}(G)$, and $\mathrm{Com}(G)$ for the power graph, enhanced power
graph, and commuting graph of $G$. One possibility is to modify these in an
obvious way, so that $\mathrm{CSPow}(G)$ is the conjugacy superpower graph
and $\mathrm{OSPow}(G)$ is the order superpower graph,
with similar terminology for the other graphs.

This notation is a bit cumbersome but is hopefully self-explanatory.

\section{Reducing nine graphs to eight}

We have defined nine graphs, but two of them turn out to be the same.

\begin{theorem}
For any finite group $G$, the order superenhanced power graph of $G$ is equal
to the order supercommuting graph of $G$.
\end{theorem}

\begin{proof}
We use the basic fact that, if a finite abelian group $G$ has exponent $m$,
then it contains an element of order $m$.
By definition, the graph
$\mathrm{OSEPow}(G)$
is a spanning subgraph of 
$\mathrm{OSCom}(G)$. We have to prove the reverse implication.
So suppose that $\{x,y\}$ is an edge of 
$\mathrm{OSCom}(G)$. By definition, there
exist elements $x'$ and $y'$ such that $o(x)=o(x')$, $o(y)=o(y')$, and 
$x'y'=y'x'$. Then $A=\langle x',y'\rangle$ is abelian. Let $m$ be its exponent,
and $z\in A$ an element of order $m$. Then $o(x')\mid o(z)$ and
$o(y')\mid o(z)$, so there exist elements $x''$ and $y''$ in $A$ with
$o(x'')=o(x')$, $o(y'')=o(y')$, and $x''$ and $y''$ are both powers of $z$.
Then $\{x'',y''\}$ is an edge of the enhanced power graph of $G$, and 
$o(x'')=o(x)$, $o(y'')=o(y)$; so $\{x,y\}$ is an edge of the order
superenhanced power graph. 
\end{proof}

Any two of the remaining eight graphs are unequal for some group $G$.
By Theorem~\ref{t:complete}, the only pairs that need to be
considered are the $\grtype$ graph and the conjugacy super$\grtype$ graph
for each of our three graph types $\grtype$; all of these are settled by the
example $G=S_3$ (the dihedral/symmetric group of order $6$).

One could ask: Is there a group $G$ for which all eight graphs are different?
If so, what is the smallest order of such a group?

A more challenging question would be, for each pair of graph types, to 
determine the groups for which the two types of graphs coincide. This has
been solved for the original three graphs, and is a non-trivial exercise.
The power graph and enhanced power graph are equal if and only if $G$ contains
no subgroup $C_p\times C_q$ for distinct primes $p$ and $q$; this condition
characterizes the so-called \emph{EPPO groups} (elements of prime power order groups),
which were determined by Brandl~\cite{brandl} using earlier work of 
Higman~\cite{higman} and Suzuki~\cite{suzuki}; see also~\cite{cm}. The
enhanced power graph and the
commuting graph are equal if and only if $G$ contains no subgroup 
$C_p\times C_p$ for a prime $p$; this condition is equivalent to saying that
all the Sylow subgroups are cyclic or (for the prime~$2$) generalized
quaternion groups and it is not too difficult to list such groups. Indeed,
all groups with cyclic or generalized quaternion Sylow $2$-subgroups have been
determined; see \cite{bc}.

\medskip

We give two more results along these lines. For the first, recall the
definition of iterated commutators in a group: $[x,y]=x^{-1}y^{-1}xy$ and
\[[x_1,x_2,\ldots,x_{n+1}]=[[x_1,x_2,\ldots,x_n],x_{n+1}]\]
for $n\ge2$. A group $G$ is \emph{nilpotent of class at most $n$} if
$[x_1,\ldots,x_{n+1}]=1$ for all $x_1,\ldots,x_{n+1}\in G$; and a group satisfies
the $n$th \emph{Engel identity}, or is \emph{$n$-Engel}, if $[y,x,\ldots,x]=1$
(with $n$ occurrences of $x$) for all $x,y\in G$. Clearly a group which is
nilpotent of class at most $n$ is $n$-Engel; the converse is false,
but it was shown by Hopkins~\cite{hopkins} and Levi~\cite{levi} independently
that a $2$-Engel group is nilpotent of class at most~$3$.

The following lemma is proved by Korhonen given in a post on \texttt{StackExchange}
~\cite{se}.

\begin{lemma}
The following statements are equivalent for a group $G$.
\begin{enumerate}
\item Every centralizer in $G$ is a normal subgroup.
\item Any two conjugate elements in $G$ commute, ie. $x^g x = x x^g$ for all $x, g \in G$.
\item  $G$ is a $2$-Engel group, ie. $[[x,g],g] = 1$ for all $x, g \in G$.
\end{enumerate}
\end{lemma}

\begin{proof}
	
	$(a) \implies (b)$: Consider $x$ in $C_G(x)$, since a normal subgroup is a union of conjugacy classes of its elements, we have $x^g \in C_G(x)$ for all $g \in G$ .
	
	$(b) \implies (c)$: Since $x^g = x[x,g]$, if $x^g$ commutes with $x$, $[x,g]$ also commutes with $x$.
	
	$(c) \implies (a)$: If $[[x,g],g] = 1$ for all $g \in G$, then according to \cite[Lemma 2.2]{kappe} we have $[x, [g,h]] = [[x,g],h]^2$. Therefore $[C_G(x), G] \leq C_G(x)$, which means that $C_G(x)$ is a normal subgroup.
\end{proof}

\begin{theorem}\label{t:com=csc}
Let $G$ be a finite group. Then the following conditions are equivalent:
\begin{enumerate}
\item 
the commuting graph of $G$ is equal to the conjugacy supercommuting graph;
\item
the centralizer of every element of $G$ is a normal subgroup of $G$;
\item
$G$ is a $2$-Engel group.
\end{enumerate}
\end{theorem}

\begin{proof}
First, we show the equivalence of (a) and (b). Suppose the commuting graph
and conjugacy supercommuting graph are equal. If $x$ and $y$ commute, they
are joined in the commuting graph, and so every conjugate of $y$ is joined to
$x$ in the conjugacy supercommuting graph, and hence also in the commuting
graph; thus $x$ commutes with every conjugate of $y$. Hence the centralizer
of $x$ is a union of conjugacy classes, so it is a normal subgroup of $G$. The
argument reverses. The equivalence of (b) and (c) follows from the above lemma.
\end{proof}

For the second result, recall that a \emph{Dedekind group} is a group in which
every subgroup is normal. Dedekind~\cite{dedekind} showed that such a group
is either abelian or of the form $Q_8\times E\times F$, where $Q_8$ is the
quaternion group of order $8$, $E$ an elementary abelian $2$-group, and $F$ an
abelian group of odd order.

\begin{theorem}\label{t:pow=csp}
For a finite group $G$, the following conditions are equivalent:
\begin{enumerate}
\item the power graph of $G$ is equal to the conjugacy superpower graph;
\item the enhanced power graph of $G$ is equal to the conjugacy superenhanced
power graph;
\item $G$ is a Dedekind group.
\end{enumerate}
\end{theorem}

\begin{proof}
We use the fact that conjugate elements have the same order. Let $x^G$ denote
the conjugacy class of $x$ in $G$.

In either the power graph or the enhanced power graph, elements of the same
order which are adjacent generate the same cyclic subgroup. So, if either (a)
or (b) holds, then all elements of $x^G$ generate $\langle x\rangle$. So 
every cyclic subgroup of $G$, and hence every subgroup, is normal; that is,
$G$ is a Dedekind group.

For the converse, if $x$ and $y$ are joined in either the power graph or the
enhanced power graph, then there is a cyclic group $C$ containing $x$ and $y$.
If $G$ is a Dedekind group, then $C$ is normal; so $x^G\cup y^G\subseteq C$.
In the case of the enhanced power graph, $C$ is a clique, so all vertices in
$x^G$ are joined to all vertices in $y^G$, and (b) holds. In the case of the
 power graph, (a) holds because either $o(x)$ divides $o(y)$ or
\emph{vice versa}. 
\end{proof}

\begin{rem}
	There are some other elementary observations. The group $G$ has the property
	``two elements are conjugate if and only if they are equal'' exactly when $G$
	is abelian. So we have that the conjugacy super$\grtype$ graph is equal to the
	$\grtype$ graph if $G$ is abelian, for any graph type $\grtype$. In a similar
	way, the conjugacy super$\grtype$ graph is equal to the order super$\grtype$
	graph if $G$ is a group in which any two elements of the same order are
	conjugate. (There are only three finite groups with this property, the
	symmetric groups of degrees $1$, $2$ and $3$: see Fitzpatrick
	\cite[Theorem 3.6]{fitzp}.) 
\end{rem}

\section{Completeness and dominant vertices}

In this section, we begin a study of how the properties of the graphs relate to
properties of the groups they are built on.

\subsection{When is the $\partype$ super$\grtype$ graph complete?}\label{complete}

The theorem below summarises the answer to this question ``When is the graph
complete?'' for our three types of graph and three types of partition, and is
intended as an example of treating the hierarchy uniformly. In the table,
$(*)$ means that the group $G$ has an element whose order is the exponent $m$
of $G$; equivalently, the \emph{spectrum} of $G$ (the set $\pi^*(G)$ of orders
of elements of $G$, sometimes denoted by $\pi_e(G)$) is
the set of all divisors of $m$. Such groups are not so rare. Any
nilpotent group has this property; and, for any finite group $G$, there is
a positive integer $r$ such that $G^r$ has property $(*)$. For example,
$(A_5)^3$ contains elements of order $30$, which is the exponent of the group.

\begin{theorem}
The following table describes groups whose power graph, enhanced power graph,
commuting graph, or their conjugacy or order supergraph is complete.
\begin{center}
\begin{tabular}{|c||c|c|c|}
\hline
& power graph & enhanced    & commuting \\
&             & power graph & graph \\
\hline\hline
equality & cyclic    & cyclic & abelian \\
      & $p$-group &        &         \\
\hline
conjugacy & cyclic & cyclic & abelian \\
          & $p$-group &     &         \\
\hline
order & $p$-group & $(*)$    & $(*)$  \\
\hline
\end{tabular}
\end{center}
\label{t:complete}
\end{theorem}

\begin{proof}
The results for the power graph, enhanced power graph, and commuting graph
are well-known \cite{gong}.

To prove that the conjugacy supercommuting graph is complete if and only if
$G$ is abelian, we use a result which goes back to Jordan~\cite{jordan}: if
$G$ is a finite group and $H$ a proper subgroup of $G$, then there is a
conjugacy class in $G$ which is disjoint from $H$. (For the reader's convenience
we sketch a proof. Let $H$ have index $n$ in $G$, and consider the action of
$G$ by right multiplication on the set of right cosets of $H$. This action is
transitive, so by the Orbit-counting Lemma (sometimes called Burnside's Lemma),
the average number of fixed points of the elements of $G$ is~$1$. But the
identity fixes $n$ points, and $n>1$; so some element $g$ fixes no point. This
means that $g$ lies in no conjugate of $H$; equivalently, no conjugate of $g$
lies in $H$. For a modern take on Jordan's theorem, we strongly recommend a
paper of Serre~\cite{serre}.)

Now, if $G$ is abelian, then the conjugacy
supercommuting graph coincides with the commuting graph, and is complete.
So suppose that $G$ is a finite group whose conjugacy supercommuting graph is
complete, and take any element $g\in G$. If $C_G(g)\ne G$ then, by Jordan's
theorem, there is an element $h$ such that the conjugacy class of $h$ is
disjoint from $C_G(g)$; thus no conjugate of $h$ commutes with $g$, and
so $g$ and $h$ are non-adjacent, a contradiction. Thus $C_G(g)=G$, or
$g\in Z(G)$. Since this holds for all $g\in G$, we see that $G$ is abelian.

If $G$ is not a $p$-group, then it contains elements of distinct prime
orders. These elements are non-adjacent in the power graph and both of its
supergraphs. So if any of these graphs are complete, then $G$ must be a
$p$-group. Conversely, if $G$ is a $p$-group, its order superpower graph is
complete, as shown in~\cite{ls:rdgg}.

If $G$ is a cyclic $p$-group, then its power graph, and hence its conjugacy
superpower graph, is complete. Suppose conversely that $G$ is a group whose
conjugacy superpower graph is complete. Then $G$ cannot have elements of 
distinct prime order, so $G$ is a $p$-group. Let $g$ be an element of
order $p$ in $Z(G)$. Then $G$ is conjugate only to itself, so cannot be
joined to any element of order $p$ outside $\langle g\rangle$; so there can
be no such elements. Thus $G$ has a unique subgroup of order $p$, and by a
result of Burnside~\cite[Sections 104--105]{burnside} it is cyclic or
generalized quaternion. But generalized quaternion groups contain
non-conjugate subgroups of order $4$, so cannot arise here.

If $G$ is cyclic, then its enhanced power graph, and hence its conjugacy
superenhanced power graph, is complete. Suppose conversely that $G$ is a group
whose conjugacy superenhanced power graph is complete. Then the conjugacy
supercommuting graph of $G$ is complete, so $G$ is abelian. Then the conjugacy
superenhanced power graph coincides with the enhanced power graph, so $G$ is
cyclic.

Finally, let $G$ be a group whose order
supercommuting graph is complete. Take two elements of $G$, with orders
(say) $g$ and $h$. Since $g$ and $h$ are adjacent, we can replace them with
elements of the same orders which commute, and so the
order of their product is the least common multiple of $k$ and $l$. Thus
the set $\pi_e(G)$ is closed under taking least common multiples (as well as
under taking divisors), and so $(*)$ holds. Conversely, if $g\in G$ has order
equal to the exponent of $G$, then every element in $\pi_e(G)$ is the order
of some power of $g$, and all these powers are joined in the enhanced power
graph and in the commuting graph. So the order supercommuting graph is
complete. 
\end{proof}

\subsection{Dominant vertices}

A graph is complete if and only if every vertex is \emph{dominant} (or
\emph{universal}), that is, joined to all other vertices. So, as a
generalization of Theorem~\ref{t:complete}, we could ask: for each of the nine
graphs, which elements of a group $G$ are dominant vertices? The answers are
known for the basic graphs, and can be found in \cite[Section 9.1]{gong}; we
summarise the results here.
\begin{enumerate}
\item The set of dominant vertices of the power graph of $G$ is the whole of
$G$, if $G$ is a cyclic $p$-group; the identity and the generators of $G$, if
$G$ is cyclic but not a $p$-group; the centre, if $G$ is a generalized 
quaternion group; and only the identity in all other cases.
\item The set of dominant vertices in the enhanced power graph is a cyclic
subgroup of $Z(G)$ called the \emph{cyclicizer} of $G$; it is the product of
the Sylow $p$-subgroups of $Z(G)$ for those primes $p$ for which a Sylow
$p$-subgroup of $G$ is cyclic or generalized quaternion.
\item The set of dominant vertices in the commuting graph is the centre
$Z(G)$.
\end{enumerate}

Now we solve the problem for the conjugacy supergraphs.

\begin{theorem}
If $\grtype$ is the power graph, enhanced power graph, or commuting graph,
then the set of dominant vertices in the conjugacy super$\grtype$ graph of
$G$ is the same as the set of dominant vertices in the $\grtype$ graph.
\end{theorem}

\begin{proof}
We show this first for the commuting graph. Suppose that $g\in G$ and $g$ is
joined to all other vertices of $G$ in the conjugacy supercommuting graph.
By Proposition~\ref{p:orbit}, $g$ is joined to an element of every conjugacy
class of $G$; in other words, its centralizer $C_G(g)$ meets every conjugacy
class. Hence $C_G(g)=G$, so $g\in Z(G)$. Thus $g$ is joined to all other
vertices in the commuting graph.

Now suppose that $\grtype(G)$ is the power graph or enhanced power graph of $G$,
and let $g$ be a dominant vertex in the conjugacy super$\grtype$ graph of $G$.
Since the conjugacy super$\grtype$ graph is a spanning subgraph of the
conjugacy supercommuting graph, $g\in Z(G)$. So, for any $h\in G$, $h$ is joined
to a conjugate of $g$. But the only conjugate is $g$ itself; so $g$ is joined
to all other vertices in $\grtype(G)$.
\end{proof}

We have also solved the problem for the order superpower graph. If $G$ has
prime power order then its order superpower graph is complete; so we can
suppose not.

\begin{prop}
Let $G$ be a group not of prime power order, having exponent $m$. Then the
set of dominant vertices in the order superpower graph of $G$ consists of the
identity and the elements of order $m$ (if any).
\end{prop}

\begin{proof}
Let $p_1,\ldots,p_r$ be the prime divisors of $|G|$ with $r>1$, and let
$p_i^{a_i}$ be the largest power of $p^i$ dividing the order of an element of
$G$; then there are elements of order $p_i^{a_i}$ in $G$. Suppose that $n$ has
the property that elements of order $n$ are dominant, and that $n>1$. Then,
for each $i$, either $p_i^{a_i}$ divides $n$, or $n$ divides $p_i^{a_i}$.
Since $r>1$, the second cannot hold. (If, say, $n$ divides $p_1^{a_1}$, then
$n$ is a proper power of $p_1$; then neither $p_2^{a_2}\mid n$ or
$n\mid p_2^{a_2}$ can hold.) Thus $n$ is the product of the prime powers
$p_i^{a_i}$, which is equal to the exponent of $G$.

Conversely, if $g$ has order $m$, the exponent of $G$, then $o(h)\mid o(g)$
for all $h\in G$, so $g$ is dominant.
\end{proof}

We have not found a characterisation of the dominant vertices in the
(one) remaining case, the order supercommuting graph.

\paragraph{Remark} If a graph has a dominant vertex, then it is connected,
with diameter at most~$2$. So it is customary to remove the dominant vertices
in order to get non-trivial questions about connectedness and diameter. This
is one reason why it is important to know such vertices. This remark suggests
that a next step in the investigation of these graphs would be to decide about the
connectedness of the ``reduced'' graphs.

\section{Some graph properties and parameters}

In this section, we discuss several further graph properties (perfectness,
universality) and parameters (clique number) for our graphs.

\subsection{Perfectness and universality}

It is known that power graphs of finite groups are perfect, but enhanced power
graphs and commuting graphs are not necessarily perfect; indeed, any finite
graph can be embedded as an induced subgraph in the enhanced power graph (or
the commuting graph) of a finite group (see~\cite{gong}). In this section we
give some similar results for supergraphs.

\begin{theorem}
The conjugacy or order superpower graph of a finite group is the comparability
graph of a finite partial preorder and hence is perfect.
\end{theorem}

\begin{proof}
For this, we define directed versions of these graphs and show that they are
comparability graphs. For conjugacy, we put an arc from $x$ to $y$ if some
conjugate of $y$ is a power of $x$ (or equivalently if $y$ is a power of some
conjugate of $x$); for order, we put an arc from $x$ to $y$ if $o(y)\mid o(x)$.
Both are reflexive (if we add loops), symmetric and transitive. 
\end{proof}

\begin{theorem}
Every finite graph $\Gamma$ is embeddable as an induced subgraph in the
conjugacy superenhanced power graph, and in the conjugacy supercommuting
graph, of some finite group.
\end{theorem}

\begin{proof}
The proof in \cite[Theorem 5.5]{gong} of the analogous result for the enhanced
power graph constructs an abelian group, where conjugacy coincides with
equality, proving the result for this case. Here, we give a different proof,
which works for both graph types. We use the fact that two elements of 
distinct prime orders are joined in the enhanced power graph if and only if
they are joined in the commuting graph. (One way round is trivial since the
enhanced power graph is a subgraph of the commuting graph. In the other
direction, if $g$ and $h$ have distinct prime orders and commute, then both
are powers of $gh$.) Now, if $\Gamma$ is a complete graph on $n$ vertices, then
we can take $G$ to be the direct product of cyclic groups of distinct prime
orders $p_1,\ldots,p_n$; if $X$ consists of one element of each prime order,
then the induced subgraph of the commuting graph of $G$ on $X$ is $\Gamma$.
So we may assume that $\Gamma$ is not complete.

First, we observe that it is possible to find a set of $n$ prime numbers
$p_1,\ldots,p_n$ such that, if $i\ne j$ and $p_i<p_j$, then $p_i\mid p_j-1$.
This is proved by induction. Suppose that $p_1,\ldots,p_{n-1}$ have been
chosen. Then we choose $p_n$ to be congruent to $1$ mod $p_i$ for 
$i=1,\ldots,n-1$ (this is possible by the Chinese remainder theorem) and
to be prime (this is possible by Dirichlet's theorem on primes in arithmetic
progression).

Now for $i,j\in\{1,\ldots,n\}$, let $G_{ij}$ be the direct product of the
non-abelian group of order $p_ip_j$ and the cyclic group of order $p_k$ for
every $k\notin\{i,j\}$; let $x_{ijk}$ be an element of order $p_k$ in
$G_{ij}$ for $k=1,\ldots,n$. We note that the induced subgraph of the
commuting graph of $G_{ij}$ on $\{x_{ij1},\ldots,x_{ijn}\}$ is the complete
graph $K_n$ with the edge $\{i,j\}$ deleted.

Given a graph $\Gamma$ with vertex set $\{1,\ldots,n\}$, let $G$ be the
direct product of the groups $G_{ij}$ over all \emph{nonedges} $\{i,j\}$
of $G$; let $x_k$ be the element of $G$ which projects onto $x_{ijk}$ in the
factor $G_{ij}$ for all such pairs $\{i,j\}$. Then the element $x_k$ has
order $p_k$, and $x_k$ and $x_l$ commute if and only if $x_{ijk}$ and
$x_{ijl}$ commute for all $\{i,j\}$, that is, $\{k,l\}$ is an edge of $\Gamma$.

Finally, we note that, if two elements commute, they are joined in the
conjucacy supercommuting graph. Conversely, if $x_j$ and $x_k$ do not commute,
then they project onto non-commuting elements in $G_{jk}$; the structure of
the non-abelian group of order $p_jp_k$ shows that conjugates of these elements
also do not commute.
\end{proof}

\subsection{Clique number}

We describe the maximal cliques in the order superpower and superenhanced
power graphs.

We begin with some definitions. The \emph{spectrum} $\\pi^*(G)$ of a finite
group $G$ (cf. Section \ref{complete}) is closed
under divisibility (if $k\in\pi^*(G)$ and $l\mid k$ then $l\in\pi^*(G)$),
so $\pi^*(G)$ is determined by the set $\pi^*_{\max}(G)$ of its elements
which are maximal in the divisibility partial order.

A sequence $(m_1,m_2,\ldots,m_r)$ of distinct positive integers is a
\emph{chain} in the divisibility partial order if $m_i\mid m_{i+1}$ for
$i=1,\ldots,r-1$. It is a \emph{maximal chain} if $m_1=1$ and $m_{i+1}/m_i$
is prime for $i=1,\ldots,r-1$. The \emph{top} of the chain is $m_r$.

For a finite group $G$ and $m\in\pi^*(G)$, we let $G(m)$ denote the set of
all elements of order $m$ in $G$.

\begin{theorem}
\begin{enumerate}
\item A maximal clique in the order superpower graph has the form 
$G(m_1)\cup G(m_2)\cup\cdots\cup G(m_r)$ for some maximal chain whose top
belongs to $\pi^*_{\max}(G)$.
\item A maximal clique in the order superenhanced power graph has the form
$\displaystyle{\bigcup_{r\mid m}G(r)}$, for some $m\in\pi^*_{\max}(G)$.
\end{enumerate}
\end{theorem}

Given this theorem, the clique numbers of these graphs are obtained by
maximizing over all maximal chains with top in $\pi^*_{\max}(G)$ (in the
first case) or elements of $\pi^*_{\max}(G)$ (in the second).

\begin{proof}
(a) Given a clique $C$ in the order superpower graph, let $m$ be the
largest element of $C$. Then the orders of all other elements of $C$ divide
$m$, and so they form a chain with top $m$. So $C$ is contained in the
union given in part (a) of the theorem, and maximality implies that the chain
is maximal, its top is in $\pi^*_{\max}(G)$, and that every element of
$G(k)$ for $k$ in the chain belongs to $C$.

\medskip

(b) Given $k$ and $l$, elements of $G(k)$ and $G(l)$ are joined in the order
superenhanced power graph if and only if $\lcm(k,l)\in\pi^*(G)$. So the
set of orders of elements in a clique has a unique maximal element $m$. As
in the preceding argument, if $C$ is maximal, then the orders include every
divisor of $m$ and $C$ contains $G(k)$ whenever $k\mid n$.
\end{proof}

For the superpower and superenhanced power graphs, we do not give a formula,
but explain what maximal cliques look like.

\begin{prop}
Let $G$ be a finite group.
\begin{enumerate}
\item Let $C$ be a maximal clique in the conjugacy superenhanced power graph
of $G$. Then there exists $m\in\pi^*(G)$ such that $C$ is the union of
a conjugacy class of cyclic subgroups of order~$m$.
\item Let $C'$ be a maximal clique in the conjugacy superpower graph of $G$.
Then there exists $m\in\pi^*(G)$ and a maximal chain $(m_1,\ldots,m_r)$
of divisors of $m$ and a conjugacy class of cyclic subgroups of $G$ of 
order~$m$ (with union $C$) such that $C'$ consists of all elements of $C$
which have order $m_i$ for some $i$ with $1\le i\le r$.
\end{enumerate}
\end{prop}

\begin{proof}
In either case, let $m$ be the largest order of an element of the clique;
then the order of any element of the clique divides $m$. (This is clear
for the power graph. For the enhanced power graph, let $g$ be an element of
order $m$. If there are elements of order $q$ not dividing $m$, then there is
one (say $h$) joined to $g$ in the enhanced power graph; but then there is
an element of larger order to which both $g$ and $h$ are joined, and so it
is joined to all elements of the clique. Now the result follows as in the
preceding theorem. 
\end{proof}

Note that we cannot conclude in this case that $m\in\pi^*_{\max}(G)$. For
example, in the dihedral group $D_4$ of order $8$, the cyclic group $C_4$
consisting of rotations is a maximal clique in either graph; the reflections
fall into two conjugacy classes, each of which (together with the identity)
forms a maximal clique. In general, it seems not an easy task to decide which
clique (as given in the Proposition) is largest, or to give a formula for its
size.

\section{Open problems and further directions}

We mention here some questions which have arisen in this investigation which 
we have not been able to answer, and some further directions for research.

\begin{prob}
Extend these investigations to other graphs defined on groups (such as the
nilpotency and solvability graphs) and other equivalence relations (such as
automorphism conjugacy). This question for the conjugacy supergraphs has been
studied in several papers, for example, \cite{me,nbc}.
\end{prob}

\begin{prob}
Complete the characterization of the classes of groups $G$ for which a given
pair of the super graphs on $G$ coincide, especially for classes that are
adjacent in a row or column of the $3\times 3$ table (as in
Theorem~\ref{t:complete}): see Theorems~\ref{t:com=csc} and~\ref{t:pow=csp}.
\end{prob}

\begin{prob}
Characterize the dominant vertices in the order supercommuting graph.
\end{prob}

\begin{prob}
What can be said about the connectedness of super graphs when dominant vertices
are deleted?
\end{prob}

\begin{prob}
Characterize the cliques of  maximum size in the conjugacy supergraphs.
\end{prob}

\end{document}